\documentclass[12pt]{amsart}
\usepackage{amssymb}
\usepackage{amscd}
\usepackage{mathrsfs}
\usepackage{a4wide}

\def\cal#1{\mathscr{#1}}
\newcommand{\lo}{\longrightarrow}
\newcommand{\re}{{\Bbb R}}
\newcommand\AR{\ensuremath{\mathsf{AR}}}

\makeatletter
\def\diam{\mathop{\operator@font diam}\nolimits}
\def\st{\mathop{\operator@font st}\nolimits}
\def\cone{\mathop{\operator@font cone}\nolimits}
\def\mesh{\mathop{\operator@font mesh}\nolimits}
\makeatother

\newtheorem{theorem}{Theorem}[section]
\newtheorem{corollary}[theorem]{Corollary}

\newtheorem{proposition}[theorem]{Proposition}
\theoremstyle{definition}
\newtheorem{definition}[theorem]{Definition}

\theoremstyle{remark}

\numberwithin{equation}{section}

\begin{document}\title{A short proof of Toru\'nczyk's Characterization Theorems}
\author{Jan J. Dijkstra}
\address{Faculteit der Exacte Wetenschappen\,/\,Afdeling Wiskunde
\\Vrije Universiteit\\
De Boelelaan 1081\\
1081 HV \ Amsterdam\\
The Netherlands}
\email{jan.dijkstra1@gmail.com} 
\author{Michael  Levin}

\address{Department of Mathematics\\
Ben Gurion University of the Negev\\
P.O.B. 653\\
Be'er Sheva 84105, Irael}\email{mlevine@math.bgu.ac.il}

\author{Jan van Mill}
\address{Faculteit der Exacte Wetenschappen\,/\,Afdeling Wiskunde
\\Vrije Universiteit\\
De Boelelaan 1081\\
1081 HV \ Amsterdam\\
The Netherlands}

 \email{j.van.mill@vu.nl}

\date{\today}



\thanks{The second author was supported in part by a grant from the  Netherlands Organisation for Scientific Research (NWO)
and ISF grant 836/08}

\begin{abstract}
We present short proofs of Toru\'nczyk's well-known characterization theorems of the Hilbert cube and Hilbert space, respectively.
\end{abstract}

\maketitle

\section{Introduction}
\emph{All spaces under discussion are assumed to be separable and metrizable.}

Recall that a compactum (complete space) $Y$ is
\emph{strongly universal} if any map $f : X \lo Y$ from
a compactum (complete space) can be
approximated arbitrarily closely by a (closed) embedding into $Y$. In the non-compact
case the closeness is measured by open covers of $Y$.
A compactum (complete space) is said to be a \emph{Hilbert type compactum}
(\emph{Hilbert type space}) if it is strongly universal and an \AR.

\begin{definition}
\label{definition}
We say that  a Hilbert type compactum (space) $H$ is a \emph{model space}
for Hilbert type compacta (Hilbert type spaces)
if it has the following properties:

(i) (stability) $H\approx H\times [0,1]$;

(ii) (Z-set Unknotting Theorem) given an open cover $\cal U$ of $H$,
an open subset  $\Omega$ of $H$,
 homeomorphic Z-sets $Z_1$ and $Z_2$ of $H$ contained in $\Omega$ and
 a homeomorphism $\phi: Z_1 \lo Z_2$ homotopic
 to the identity map   of $Z_1$ by
 a homotopy  controlled by $\cal U$ and supported by $\Omega$
  there exists
  a homeomorphism $\Phi:  H \lo H$ such that $\Phi$ extends $\phi$ and
 $\Phi$ is  controlled by $\cal U$ and  supported by $ \Omega$.
\end{definition}

Note that  both the strong universality and condition (ii)
of Definition~\ref{definition} (the Z-set unknotting theorem) are  local properties,
that is
if they  hold locally then they hold globally as well.

The goal of this note is  to present a complete and self-contained
proof of the fact  that the existence of a model space
for Hilbert type compacta (Hilbert type spaces) implies the characterization
theorem which says that every two Hilbert type compacta (spaces) are homeomorphic.
These results, first proved by Toru\'nczyk~\cite{tor:hilbertcube,tor:hilbertspace,tor:fix}, are widely known and were applied in diverse settings. They also inspired the characterization results of the universal Menger spaces by Bestvina~\cite{bestvina} and the recent work on N\"{o}beling spaces
\cite{ageev1,ageev2,ageev3}, \cite{nagorko}, \cite{levin1, levin2}.

One would probably expect that our abstract approach
 will make the proofs  longer and more complicated and we were  surprised
to find out that this approach can considerably shorten and simplify the proofs
despite that we use already known techniques and ideas.
This was mainly achieved by  carefully analyzing  existing  proofs,
extracting essential parts, avoiding unnecessary repetitions, splitting
 the proofs into short parts and  sometimes
 reversing
the historical order of  the results. For example, we simplified the proof of Miller's 
cell-like resolution theorem \cite{miller} by using   techniques
introduced  later for proving the characterization theorems.

It was known before the characterization theorems were proved  that
  model spaces for
 Hilbert type compacta and Hilbert type spaces exist
 (for example Hilbert cube and Hilbert space respectively).
One of the features of our approach is that we never work
with a particular realization  of a model space, we don't even assume that
the Hilbert cube $Q$ is a model space for Hilbert type compacta.
Although in the compact case we are able to detect on a relatively
early stage of the proof that if a model space exists  it must be
homeomorphic to $Q$, in the non-compact case we can pretend not to know what  a model space looks like until
the characterization is proved.

\section{Preliminaries}

We assume that the reader is familiar with general facts regarding
\AR's, Hilbert type compacta and Hilbert type spaces, cell-like maps etc. Most of the necessary information can be found in \cite[\S\S7.1-7.3]{vm:book}. See also Chapman~\cite{chapman:structure} and Edwards~\cite{edwards:charac}. Earlier simplifications of the Toru\'nczyk's proofs can be found in~\cite{bbmw} and \cite{walsh:alternate}.

Let $ f : X \lo Y$ be a proper map and let $A$ a closed subset of $Y$.
By $X\cup_f A$ we denote the quotient space  of $X$ obtained by collapsing the fibers
over $A$ to singletons. As usual, for a proper surjection $f : X \lo Y$, we let $M(f)$ denote the mapping cylinder of $f$, that is, $M(f)$ is obtained from $X\times[0,1]$ by replacing $X\times\{1\}$ by $Y$. We usually let $\pi_Y : M(f) \lo Y$  denote the projection.
If $f: X\to Y$ and $A\subset Y$ then we say that $f$ is {\em one-to-one over $A$} if the restriction of $f$ to $f^{-1}(A)$ is one-to-one.

The following result displays a technique for proving the existence of a homeomorphism between an $M(f)$ and the range of $f$ that is known as `the Edwards trick'.
\begin{proposition}\label{Edwardstrick} Let $f:X\to Y$ be a proper surjection between complete spaces and let $\pi_Y:M(f)\to Y$ be the natural projection.
Assume that for every open cover $\mathcal U$ of $Y$ there are a near homeomorphism $\alpha:M(f)\to M(f)$ and a map $\beta:M(f)\to Y$, $\mathcal U$-close to $\pi_Y$, with
$\pi_Y=\beta\circ \alpha$ and $\beta$ one-to-one over $\beta(X\times\{0\})$. Then $\pi_Y$ is a near homeomorphism.
\end{proposition}
\begin{proof}
  Identify $M(f) \setminus Y$ with $X\times [0,1)$ and
let  $\pi_X : M(f)\setminus Y \lo X$ be the natural projection.
With the aim of using Bing shrinking consider      open covers  ${\cal U}_M$ of  $M(f)$  and
  ${\cal U}_Y$  of $Y$ respectively. Then there are
  an open cover  ${\cal U}_X$
  of $X$ and   an infinite sequence of continuous functions $t_i : X \lo (0,1)$ ($i\in\mathbb N$) with $t_{i+1}(a)<t_i(a)$ and $\lim_{i\to\infty} t_i(a)=0$ for every $a\in X$
   such that the collection ${\cal U}$
  of open sets in $M(f)$ defined below refines ${\cal U}_M$.  
  Assume that $[0,1]$ is located on the vertical axis as usual so that the points of  $Y$ are above the graph of every function
  $t_i$ in $M(f)$.
  The collection
  ${\cal U}$ consists of the portions of the sets of $\pi_Y^{-1}({\cal U}_Y)$
  above the graph of the function $t_3$  and the portions of the sets of
  $\pi_X^{-1}({\cal U}_X)$ between the graphs of $t_{i+2}$ and $t_i$ for  $i\in\mathbb N$.

Assume that $\alpha$ and $\beta$ are chosen as in the premise of the proposition with $\beta$ ${\cal U}_Y$-close to $\pi_Y$. Then one can find infinite sequence of continuous functions $s_i : X \lo (0,1)$ ($i\in\mathbb N$) with $s_{i+1}(a)<s_i(a)$ and $\lim_{i\to\infty} s_i(a)=0$ for every $a\in X$
  with the following
 properties.  For every $a\in X$, $s_{i+1}(a)<s_i(a)$ and $\lim_{i\to\infty} s_i(a)=0$ and
  for every fiber $F$  of $\beta$ that is not a singleton in $X \times \{0\}$ we have that
 $F$ lies in between the graphs of $s_{i+2}$ and $s_i$ for some $i\in \mathbb N$
  if $F$ intersects the closed region below the graph of $s_2$;
  $F$ lies above the graph of $s_3$
  if $F$ intersects the closed region above the graph of  $s_2$ and
 $\pi_X(F)$ is contained in an element
 of $ {\cal U}_X$ if  $F$ intersects  the closed region below the graph of $s_1 $.
 
Put $t_0(a)=s_0(a)=1$ for each $a\in X$. Consider  the homeomorphism  $\psi :X\times [0,1]\lo X \times [0,1]$
 that maps every interval $\{a\}\times [t_{i+1}(a),t_i(a)]$ linearly onto $\{a\}\times [s_{i+1}(a),s_i(a)]$. Then   $\Psi : M(f) \lo M(f)$
  is the homeomorphism of $M(f)$ induced by $\psi$.
 We have that $\pi_Y=\pi_Y \circ \Psi$
  and  the fibers of $\beta\circ \Psi$ refine ${\cal U}_M$.

  Recall that $\alpha$ is a near homeomorphism and
  $\pi_Y=\beta \circ \alpha$. Then $\alpha$ can be approximated by
  a homeomorphism $\Phi $  such that $\alpha\circ \Phi^{-1}$ is so close
  to the identity map of $M(f)$ that the fibers of
  $\pi_Y \circ  \Phi^{-1} \circ \Psi=\beta \circ\alpha \circ  \Phi^{-1} \circ \Psi$
   also refine ${\cal U}_M$.
  Clearly, if $\beta$  and $\Phi$ are  sufficiently close to $\pi_Y$
  and $\alpha$ respectively then $\pi_Y \circ \Phi^{-1} \circ \Psi$ is
  close to $\pi_Y$ and the proposition follows from Bing's shrinking criterion.
  \end{proof}

\section{Topological characterization of the Hilbert Cube}
\label{hilbert-cube-characterization}
Everywhere in this section `model space' means
`model space for Hilbert type compacta'.
We will show that the existence of a model space implies
 the characterization of Hilbert type compacta.

\begin{proposition}
\label{adjoint}
Let $H$ be a model space,
$f : H \lo Y$ be a cell-like map, $Y$ an \AR\ and  $A$ a closed subset
of $Y$ such that $Z=f^{-1}(A)$ is a Z-set in $H$. Then
the quotient map  $\pi : H \lo H\cup_f A$ is a near homeomorphism and
$A$ is a Z-set in $H\cup_f A$.
\end{proposition}

\begin{proof}
Fix $\varepsilon > 0$ and let $\cal U$ be
an open $\varepsilon$-cover of $Y$ and
$\Omega$ the $\varepsilon$-neighborhood of $Z$.
Using that $f$ is a fine homotopy equivalence,
 we lift $f$ restricted to $Z$
to a $Z$-embedding $\phi: Z \lo H$ such that $\phi(Z) \cap Z=\emptyset$,
$\phi(f^{-1}(a))$ is of diameter less than $\varepsilon$ for every $a \in A$
and  $\phi$ is homotopic
to the identity map  of $Z$ by a homotopy controlled by $f^{-1}({\cal U})$
and supported by $\Omega$. Then, by Definition~\ref{definition},  $\phi$ extends to
a homeomorphism $\Phi : H\lo H$ controlled by $f^{-1}({\cal U})$ and
supported by $\Omega$. Hence Bing's shrinking criterion implies that $\pi$ is a near
homeomorphism. Then, since $f^{-1}(A)=\pi^{-1}(A)$ is a Z-set in $H$, we get that
$A$ is a Z-set in $H\cup_f A$. Indeed, approximate the identity map of $H$ by a map $f': H \lo H\setminus Z$
and approximate $\pi$ by a homeomorphism $\pi'$.
Then the map $\pi \circ f' \circ (\pi')^{-1}$ witnesses that $A$ is a Z-set
in $H \cup_f A$.
\end{proof}

\begin{corollary}
\label{mapping-cylinder}
Let $H$ be a model space,
$f : H \lo Y$  a cell-like map and $Y$ a compact \AR. Then the projection from
$H\times [0,1]$ to $M(f)$ is a near homeomorphism.
In particular, we get that $M(f)\approx H$.
\end{corollary}

\begin{proof} Follows from Proposition~\ref{adjoint}. \end{proof}

Let $f : X \lo Y$ be a map of compacta and $\pi_Y : M(f) \lo Y$  the projection.
We say that $f$ is a $\bf nice$ map if the identity map of $Y$ can arbitrarily closely be approximated
by an embedding $g : Y \lo Y$ such that $\pi_Y^{-1}(g(Y))$ is a Z-set in $M(f)$. Note that
for every map $f : X \lo Y$ and a Hilbert type compactum $C$
the induced map $f \times  id : X \times C  \lo Y \times C$ is
a nice map since $M(f \times id)=M(f) \times C$ and $\pi_{Y \times C}=\pi_Y \times id$.
Also note that the identity map on a Hilbert type compactum is nice.

\begin{proposition}
\label{edwards}
Let $H$ be a model space,
 $f : H \lo Y$ be a nice cell-like map and $Y$  a Hilbert type compactum. Then the projection
$\pi_Y : M(f) \lo Y$ is a near homeomorphism.
\end{proposition}
\begin{proof} We aim to use Proposition~\ref{Edwardstrick}.
Since $Y$ is a Hilbert type compactum and $f$ is a nice map we conclude that
$\pi_Y$ restricted to $H\times \{0\}$ can be arbitrarily closely approximated
by an embedding $g : H\times \{0\} \lo  Y$ such that for $A=g(H\times \{0\}) $
we have that
$\pi_Y^{-1}(A)$ is a Z-set in $M(f)$.
Let $\pi : M(f) \lo M(f) \cup_{\pi_Y} A$ and $\pi'_Y :  M(f) \cup_{\pi_Y} A \lo Y$
be the natural projections. By Corollary~\ref{mapping-cylinder} and
Proposition~\ref{adjoint} the projection
$\pi$ is a near homeomorphism and $A$ is a Z-set in $M(f) \cup_{\pi_Y} A$.
Hence one can choose a homeomorphism $h : M(f) \lo  M(f) \cup_{\pi_Y} A $
so that  the map $\pi'_Y \circ h  : M(f) \lo Y$ is
 as close to $\pi_Y$ as we wish.  In addition,
 by Definition \ref{definition} and Corollary~\ref{mapping-cylinder} we can replace $h$ by its composition with a homeomorphism
 of $M(f)$ sending   $H \times \{ 0 \}$  to $h^{-1}(A)$. Hence we may assume
 that $h$  sends $H \times \{ 0 \}$ onto $A$. Putting $\alpha=h^{-1}\circ\pi: M(f) \lo M(f)$ and $\beta=\pi'_Y\circ h : M(f) \lo Y$ we note that $\alpha$ is a near homeomorphism,
   $\pi_Y =\beta\circ\alpha$, $\beta$ is one-to-one over $A= \beta(H\times\{0\})$,    the maps $\pi_Y$ and $\beta$ are as close as we wish. Thus Proposition~\ref{Edwardstrick} applies.
\end{proof}

  \begin{theorem}
  \label{near-homeo}
  Let $H$ be a model space. Then

  (i) $H\approx H\times Q$ and

  (ii) if  $f : H \lo Y$ is a cell-like map and $Y$ is a Hilbert type compactum
   then
  $f$ is a near homeomorphism.
  \end{theorem}
  \begin{proof}
   Since the identity map $id : H\lo H$ is a nice map we get
  by Proposition~\ref{edwards}
  that the projection $M(id)=H \times [0,1] \lo H$ is a near homeomorphism.
  Since $H \times Q$ is the inverse limit of the projection
  $H \times [0,1] \lo H$ we get that the projection $H \times Q \lo H$ is
  a near homeomorphism and hence $H \times Q\approx H$.

  For part {\em(ii)} assume first that $f:H \lo Y$ is a nice map. Then $f$ is a near homeomorphism since
  the projections $H \times [0,1] \lo H$, $H \times [0,1] \lo  M(f)$
  and $M(f) \lo Y$ are near homeomorphisms by Corollary~\ref{mapping-cylinder} and Proposition~\ref{edwards}.

  Now consider the general case. Recall that $H\approx H\times Q$ and
   $f \times id : H \times Q  \lo Y \times Q$ is a nice map, because $Q$ is a Hilbert type compactum. Then
   $f \times id$ is a near homeomorphism and hence $Y\times Q\approx H$.
   Note that
   the projections $H \times Q \lo H$ and $Y \times Q\lo Y$ are nice maps
   and therefore they
  are also near homeomorphisms. All this implies that
   $f$ is a near homeomorphism.
  \end{proof}

 \begin{theorem}
 \label{resolution}
 Any   compact AR  is a cell-like image of any model space.
 \end{theorem}
 We will prove this theorem in the next section.
 Theorems Theorem~\ref{near-homeo} and Theorem~\ref{resolution}
 immediately imply the characterization theorem for
 Hilbert type compacta.

\section{Cell-like Resolution}
\label{cell-like-resolution}
In this section we prove Theorem~\ref{resolution}. For its proof
we need the following auxiliary propositions and constructions.
Recall that   a model space means  a model space for Hilbert type compacta.
  \begin{proposition}
  \label{z-set}
  Let $H$ be a model space,
   $X$  a Z-set in $H$   and $f : X\lo H$ any map.
  Then $f$ extends to a cell-like map $H \lo H$.
  \end{proposition}
 \begin{proof} Let us first  consider
 the case when $f(X)$ is a Z-set.
  By the Z-set unknotting theorem
 we may  assume that $X$ and $f(X)$ are disjoint.
 Take any Z-set  $A\subset H$ such that
 $A$ is an AR, $f(X)\subset A$ and $X \cap A=\emptyset$.
  Consider  $f$ as a map $f : X \lo A$.
 By extending $f$ over a bigger  Z-set
 we may assume that $X$ is an AR.  Let
 $M(f)$ be the mapping cylinder of $f$ and
 $\pi_A : M(f) \lo A$ the projection. Embed
 $M(f)$ as a Z-set in $H$ so that
 $X$ and $A$ are identified with the corresponding natural subsets
 of $M(f)$. Then the adjunction space $Y=H \cup_{\pi_A} A$ is an AR,
 the projection  $\pi : H \lo Y$ is cell-like and hence,
  by Proposition~\ref{adjoint},  $\pi$ is a near homeomorphism
 and $A$ is a Z-set in $Y$.
 Thus there is a homeomorphism $g : Y \lo  H$  and by the Z-set unknotting
 theorem we may assume that $g$ sends   $A$ to $A$ by
 the identity map. Then $\pi$ followed by $g$ is
 the extension of $f$ we are looking for.

 Now consider the general case.
 Let the map  $\phi  : X \lo H \times [0,1]$ be defined  by
 $\phi(x)=(f(x), 0)$. By the previous case $\phi$ extends
 to a cell-like map $\Phi : H \lo  H \times [0,1]$. Then
 $\Phi$ followed by the projection of $H \times [0,1]$ to $H$
 is the required cell-like extension of $f$.
   \end{proof}

      Let $A\subset H$ be  a compact AR and $r : H \lo A$ a retraction.
   We assume that the mapping cylinder $M(r)$ is obtained from
   $H \times [0,1]$ by replacing $H \times \{1\}$ with $A$ and
   denote by $\pi_A : M(r) \lo A$ the projection induced by $r$
   and by $\pi_I : M(r) \lo [0,1]$ the projection to $[0,1]$.
   We can  rescale the interval $[0,1]$ to another interval $[a,b]$
   and consider $M(r)$ over $[a,b]$ assuming that the interval projection
    $\pi_I$ sends $A \subset M(r)$ to the right end point $a$.
    We will also refer to $A \subset M(r)$ as the {\em right $A$-part\/} of $M(r)$ and
    $A \times \{0\} \subset H \times \{ 0\}$ as the {\em left $A$-part\/} of $M(r)$
    and the projections $\pi_A$ and $\pi_I$ as the {\em $A$-projection} and
    the {\em interval   projection} respectively.

   By the {\em extended mapping cylinder} $E(r)$ of $r$ we mean the the union
   $ M(r)\cup H \times [1/2, 1]$ in which
   we assume that
   $M(r)$ is the mapping cylinder over $[0, 1/2]$ and
   $A\times \{1/2\}\subset H \times [1/2, 1]$ is
   identified with $A \subset M(r)$ (the right $A$-part of $M(r)$). We will call
   $H\times [1/2, 1]$ the extension part of $E(r)$.
    We define the projection $\pi_I : E(r) \lo [0,1]$
   by the projections of  $M(r)$ and  $H\times [1/2, 1]$ to
   $[0,1/2]$ and $[1/2,1]$ respectively and
   we  define
   the map  $\pi_H^E : E(r) \lo H$ as the union of  the projection
    $\pi_A : M(r) \lo A$ and  the projection  of  $H\times [1/2, 1]$ to $H$.
   Clearly we can rescale $[0,1]$ to an interval $[a,b]$ so that
    $M(r) \subset E(r)$  and the extension part of $E(r)$ will be sent
   by $\pi_I$
   to $[a,c]$ and $[c,b]$ respectively for some $ a < c <b$. In that
   case we say that $E(r)$ is the extended cylinder over $[a,b]$
   with $M(r)\subset E(r)$ being over $[a,c]$.

   \begin{proposition}
   \label{extended-cylinder}
   Let $H$ be a model space and $r : H \lo A$  a retraction. Then there is a cell-like map from
   $H$ to $E(r)$.
  \end{proposition}
  \begin{proof} By \ref{z-set} there is a cell-like map
  $\phi : H\times [1/2,1] \lo H\times [1/2,1]$ so that
   $\phi(x,1/2)=(r(x), 1/2)$ for $x \in H$.
  Consider $E(r)$ as the extended mapping cylinder over $[0,1]$ with
  $M(r)$ being over $[0, 1/2]$ and identify
  $M(r) \setminus A$ with $H \times [0,1/2)$.
  Extend $\phi$ to
  the map $\Phi : H\times [0,1] \lo E(r)$ by
  $\Phi(x,t)=(x,t)$ for $x \in H$ and $ 0\leq t<1/2$.
  Then $\Phi$ is the required cell-like map.
  \end{proof}

  Let $H$ be a model space.
  We say that a retraction $r : H \lo A$ is a {\em convenient}
  retraction if  $E(r)$ is homeomorphic to $H$. Note that
  if $H$ is a model space and $r : H \lo A $ is any retraction then
  the induced retraction $r\times id :  H \times Q \lo A \times Q$
  is a convenient retraction because $E(r \times id)= E(r) \times Q$
  is a Hilbert type compactum and
  because, by Proposition~\ref{extended-cylinder}, $E(r\times id)$ admits a cell-like
  map from $H$ and this map is a near homeomorphism by Theorem~\ref{near-homeo}.

  Let $r : H \lo A$ be a retraction. By  the telescope $M(r,n)$ of
  $n$ mapping cylinders  of $r$ we mean the union
  $M(r,n)=M_1(r) \cup \dots \cup M_n(r)$
  where $M_i(r)$ is the mapping cylinder of $r$ over the interval
  $[t_{i-1}, t_i], t_i=i/n$  and the right $A$-part of $M_{i}$ is identified
  with the left $A$-part of $M_{i+1}(r)$ by the identity map of $A$
  for $1\leq i \leq n-1$.
  The projections of $M_i(r)$ to $A$ and $[t_{i-1}, t_i]$
  induce
  the corresponding projections $\pi_A : M(r,n)\lo A$ and $\pi_I : M(r,n)\lo [0,1]$.
  Clearly $[0,1]$ can be rescaled  to any interval $[a,b]$.

  In a similar way
  we define the infinite telescope $M(r, \infty)= \bigcup_{i=1}^\infty M_i(r)$ over
  an infinite  partition $0=t_0<t_1<t_2<\dots, t_i\lo \infty$, of
 the ray  $\re_+ = [0,\infty)$ 
 with the mapping cylinder $M_i(r)$ being  over
  the interval $[t_{i-1},t_i]$. Again the $A$-projections and the interval projections
  of $M_i(r)$ define the projections $\pi_A : M(r,\infty)\lo A$ and
  $\pi_\re : M(r, \infty) \lo \re_+$ to which we will refer as the $A$-projection
  and the $\re$-projection respectively. Note that if $r : H \lo A$ is a convenient retraction then
  $M(r,\infty)$ is homeomorphic to $H \times \re_+$. Indeed,
  assume that the cylinders of $M(r,\infty)$ are over the intervals
  $[i-1,i]$, $i=1,2\dots$. Then $\pi_\re^{-1}([0,1/2])$ is  homeomorphic
  to $H\times [0,1/2]$ and $\pi_\re^{-1}([i-1/2,i+1/2])$ is homeomorphic to the extended cylinder
  $E(r)$ which is homeomorphic to $H\approx H\times [i-1/2,i+1/2]$ since $r$ is convenient.
  Then using the Z-set unknotting theorem we can assemble all these pieces into
  a space homeomorphic to $H\times \re_+$.

  \begin{proposition}
  \label{double}
  Let $H$ be a model space and $r : H \lo A$ be a convenient retraction.
  Then there is a homeomorphism $\phi : M(r)\lo M(r,2)=M_1(r) \cup M_2(r)$
  such that $\phi$ sends the right $A$-part
   and the left $A$-part of $M(r)$
   to the right $A$-part of $M_2(r)$
  and  the left $A$-part of $M_1(r)$ respectively
   by the identity map of $A$. Moreover, $\phi$ can be chosen so that
  for the $A$-projections  $\pi_A : M(r)\lo A$ and $\pi_A^* : M(r,2) \lo A$
  the composition $\pi^*_A \circ \phi$ is as close to
  $\pi_A$ as we wish.
  \end{proposition}
  \begin{proof} Let $\pi_I : M(r) \lo [0,1]$ and $\pi_I^* : M(r,2) \lo [0,1]$
  be the interval projections. One can naturally identify
   $\pi_I^{-1}([0,2/3])$ with $H \times [0,2/3]$ and
   $(\pi_I^*)^{-1}([0,2/3])$ with the extended mapping cylinder
   $E(r)$ of $r$ over the interval $[0,2/3]$ with mapping
   cylinder of $E(r)$ being over $[0,1/2]$ and
   being identified with $M_1(r)$.
   Let the map  $\pi_H^E : E(r)\lo H$ be as defined above.
   Since $r$ is a convenient retraction and
   $\pi^E_H$ is cell-like, we get that $\pi^E_H$ is a near homeomorphism (Theorem~\ref{near-homeo}).
   Since the projection $\pi_H : H \times [0,2/3] \lo H$ is also
   a near homeomorphism one can find  a homeomorphism
   $\psi : H\times [0, 2/3] \lo E(r)$ so that
   $\pi_H^E \circ \psi $ is  as close to $\pi_H$ as we wish.
   Since $H\times \{ 2/3\}$  and the left $A$-part
   of $M_1(r)$ are   Z-sets
   in $E(r)$ we can, in addition, assume
   by the Z-set unknotting theorem that that
   $\psi$ sends $A \times \{0\}$ and $H \times \{2/3\}$  to
   the left $A$-part of $M_1(r)$  and $H \times \{2/3\}$
    respectively by the identity maps. Extending $\psi$ over $H \times [0,1]$
    by the identity map between $\pi_I^{-1}([2/3,1])\subset M(r)$
    and $(\pi^*_I)^{-1}([2/3, 1])\subset M_2(r)$
    we get the required homeomorphism $\phi : M(r) \lo M(r,2)$.
    \end{proof}

  \begin{proposition}
  \label{cone-resolution}
  Let $H$ be a model space and $r : H \lo A$  a convenient retraction.
  Then there is a cell-like map from $\cone (H)$ to $\cone (A)$.
  \end{proposition}
  \begin{proof}
  Denote by
  $M^n$ the infinite  telescope $M^n=\bigcup_{i=1}^\infty M^n_i$ of the mapping cylinders $M^n_i$
  of $r$
  over the intervals $[\frac{i-1}{2^n},\frac{i}{2^n}]$
  and
   let $\pi^n_A : M^n \lo A$ and $\pi^n_{\re} : M^n \lo \re_+$ be
   the $A$-projection and the  $\re$-projection of $M^n$.
   We are going to construct a cell like map $\phi : M^0 \lo A\times \re_+$.
 Let ${\cal E}_n,n=1,2,\dots,$ be a sequence of  open covers of $A$ such that
  $\mesh({\cal E}_n) < 1/2^n$,   $\st ({\cal E}_{n+1})$ refines ${\cal E}_n$
  and every set of ${\cal E}_{n}$   can be homotoped to
  a point inside a set of $\diam < 1/2^{n}$.
  By Proposition~\ref{double} take  a homeomorphism $\psi^n_{n+1} : M^n \lo M^{n+1}$
  sending each  mapping cylinder $M^n_i$ of $M^n$ to
  two consecutive mapping cylinders $M^{n+1}_{2i-1}$ and $M^{n+1}_{2i}$
  of $M^{n+1}$ as described in Proposition~\ref{double} and so that
  $\pi_A^{n+1} \circ \psi^n_{n+1}$ is ${\cal E}_{n+3}$-close to $\pi^{n}_A$.
  Denote  $\psi^n_m =\psi^{m-1}_m \circ \dots \circ \psi^n_{n+1} : M^n \lo M^m$
  for $m >n$, $\psi^n_n=id$,
  $\phi^n_A =\pi^n_A \circ \psi^0_n$ and $\phi^n_\re=\pi^n_\re \circ \psi^0_n$.
  Note that if  $x \in M^n$ belongs to $M^n_i$ then
  $ \phi^m_\re (x)) \in [\frac{i-1}{2^n},\frac{i}{2^n}]$ for every $m\geq n$.
  Thus we have that the sequences of maps $\{\phi^n_A\}$ and
  $\{ \phi^n_\re \}$
  converge and their limits we denote by  $\phi_A : M^0 \lo A$ and $ \phi_\re : M^0 \lo \re_+$ respectively.
  Consider $\phi=(\phi_A, \phi_\re) : M^0 \lo A \times \re_+$.

  We will show that $\phi$ is cell-like. Take $x=(a,t) \in A \times \re_+$ and let
  $F=\phi^{-1}(x)$. Fix $n$ and note that   $\psi^0_n(F)$ is always contained in
  at most two consecutive mapping cylinders $M^n_i$ and $M^n_{i+1}$ of $M^n$.
  Also note that $\phi_A^n(F) $ refines ${\cal E}_{n+1}$.  Then
  $\psi^0_n(F)$ can be homotoped to a point of the right  $A$-part of $M^n_i$
  inside the set $(\pi^n_A)^{-1}(B) \cap (M^n_i\cup M^n_{i+1})$ where
  $B\subset A$ is the closed $1/2^n$-neighborhood of $\phi_A^n(F)$
  in $A$. Note $\diam B  <   3/2^n$.
  Then $\diam \phi_A((\pi^n_A)^{-1}(B)) <5/2^n$
  and hence $F$ can be homotoped to a point inside the set
   $\phi^{-1}( C \times [\frac{i-1}{2^n},\frac{i+1}{2^n}])$ where $C$ is
   the closed $5/2^n$-neighborhood of $a$ in $A$. This implies that
   $\phi$ is cell-like.

   It is also clear that $\phi$ is a proper map. Recall that $M^0$ is homeomorphic
   to $H \times \re_+$. Identify $H\times \re_+$ and $A \times \re_+$
   with the complements of the vertices of $\cone (H)$ and $\cone (A)$ respectively.
   Then $\phi$ extends to a cell-like from $\cone (H) $ to $\cone (A)$
   by sending the vertex of $\cone (H)$ to the vertex of $\cone (A)$ and we are done.
 \end{proof}

\begin{proof}[Proof of Theorem~\ref{resolution}.] Let $A$ be  a compact AR  and $H$  a model space.
  Since $\cone (H)$ is the mapping cylinder of
  a constant map we get by Corollary~\ref{mapping-cylinder} that
  $H\approx \cone H$. Recall that  $A\times Q$ is a convenient retract of $H\approx H \times Q$.
  Then, by Proposition~\ref{cone-resolution},  $\cone(A \times Q)$ is a cell-like
  image of $H$ and hence, by Theorem~\ref{near-homeo}, $Q\times \cone(A\times Q)$ is
  homeomorphic to $H$. Thus $Q \times A \times Q \times [0,1]=A \times Q$ is an $H$-manifold.
  Since the Z-set unknotting theorem  is a local property
  we get that $A \times Q$ satisfies both conditions (i) and (ii) of
  \ref{definition}. Clearly $A\times Q$ is a Hilbert type compactum
   and hence $A \times Q$ is a model space.
  Then $\cone(A \times Q)\approx A \times Q$. Thus
  $H\approx Q\times \cone(A\times Q)\approx A\times Q$ and
  the projection
  of $A \times Q$ to $A$ is the  cell-like map we need.
\end{proof}

  \section{Topological characterization of Hilbert Space}\label{hilbert-space-characterization}
  The results of the previous sections were intentionally presented in
  such a way that they apply with minor clarifications
  for the characterization of Hilbert type spaces. In this section
  we  describe
  the adjustments in  the proofs  needed  in the  non-compact  setting.
 Everywhere we replace Hilbert type compacta by Hilbert type spaces,
  a model space will mean a model space for the Hilbert type spaces and
  a cell-like map will mean a proper cell-like map.
 Almost all the proofs  in Sections \ref{hilbert-cube-characterization}
 and \ref{cell-like-resolution} will work in the non-compact  setting
 with obvious trivial adjustments.
 So we will point out only those places  that require  clarifications.

\smallskip

\noindent {\bf Section \ref{hilbert-cube-characterization}.}

\smallskip

\noindent {\bf Nice Maps}. The property that for a  map $f : X \lo Y$
 the induced map $f\times id : X \times Q \lo Y \times Q$ is nice  remains
 true for  non-compact spaces $X$ and $Y$    provided $f$ is proper.
 Indeed, one can easily show that  any proper map $g=(g_Y, g_Q) : X \lo Y \times Q$
 can be arbitrarily closely approximated by

  (*)  a map $g'=(g'_Y,g'_Q) : X \lo Y\times Q$
  such that $g'_Y=g_Y$ and  $g'(X)\subset Y \times B(Q)$;

  (**) an embedding $g'=(g'_Y,g'_Q) : X \lo Y\times Q$ such that
  $g'_Y =g_Y$ and $g'(X) \subset Y \times (Q \setminus B(Q))$.

 Note that the maps $g'$ in (*) and (**) are proper
 (and  hence closed)  since $g'_Y=g_Y$ and $g$ is proper.
 Thus, by (**), the identity map of $Y\times Q$ can be arbitrarily closely
 approximated by a closed embedding $g' : Y \times Q \lo Y \times Q$
 such that $A=g'(Y \times Q) \subset Y \times (Q\setminus B(Q))$.
 Then $B=(f \times id)^{-1}(A)\subset X \times (Q \setminus B(Q))$
 and  hence, by (*), $B$ is a Z-set in $X \times Q$. Thus $f \times id$ is nice.

\smallskip

\noindent {\bf  Theorem~\ref{resolution}}.  The phrase compact AR should be replaced by the phrase Hilbert type space.
  The proof of this theorem is considered below  (clarifications to section \ref{cell-like-resolution}).

\smallskip

\noindent {\bf Section \ref{cell-like-resolution}}.
  In Section \ref{cell-like-resolution} we need to consider only
  proper maps and retractions and everywhere assume that $A$ is a Hilbert type space.

\smallskip

\noindent {\bf Proper retractions.} A proper retraction to a Hilbert type space $A$
   always exists. Indeed, let $A \subset X$ be a closed subset of a complete
   space $X$. Note that $A \times [0,1]$ is also a Hilbert type space
   with $A \times \{ 0 \}$ being a Z-set in $A \times [0,1]$.
   Then  the identity map $A \lo A \times  \{0 \}$ extends to
   a  Z-embedding $f : X \lo A \times [0,1]$ and $f$ followed  by the projection
   of $A \times [0,1]$ to $A$ provides a proper retraction from $X$ to $A$.

\smallskip

\noindent {\bf Convenient retractions.} The property that
  a retraction  $ r : H \lo A$  from
  a model space  $H$ to  $A \subset H$ induces a convenient
  retraction
   $r \times id : H\times Q \lo A \times Q$ remains true
   for non-compact spaces provided $r$ is proper.
   Indeed,  by Proposition~\ref{extended-cylinder}, there is a proper cell-like map
   $f : H \lo  E(r\times id)$. In order to show that
   $E(r\times id)$  is homeomorphic to $H$, it is enough to show, by Theorem~\ref{near-homeo},
   that  $E(r\times id) $ is strongly universal. Take any map
   $g : X \lo E(r\times id)$ from a complete  space $X$. Since
    $f$ is a fine homotopy  equivalence and $H$ is a Hilbert type space
     we can lift $g$ to a closed embedding $h : X \lo H$ so that
     $f\circ h$ is arbitrarily close to $g$. Note that
     $E(r \times id)=E(r)\times Q$. Then by (**) there is an  arbitrarily close
     approximation of  $f$ by a closed embedding $f' : H \lo E(r\times id)$
    and the closed embedding $f'\circ h : X \lo E(r \times id)$
    witnesses the strong universality of $E(r \times id)$.

\smallskip

\noindent {\bf  Proposition~\ref{cone-resolution}.}
    We assume that  metrics in $A$ and $M^0$ are  complete. The properness of $\phi$
    can be acheived as follows. Note that $\phi^n_A$ is proper on each cylinder
    $M^0_i$  of $M^0$. Then we can additionally  assume  that $\phi^{n+1}_A$ is so close
    to $\phi^n_A$ that every fiber of $\phi^{n+1}_A|M^0_i$ is contained in
    the $1/2^{n+1}$-neighborhood of the corresponding fiber of $\phi^n_A|M^0_i$ for every $i$. Take a compact
    set $K$ in $A\times \re_+$ and let  a compact set $K_A \subset A$ and
    a natural number $k $  be so that $K \subset K_A \times [0,k]$.
    Then for every $n$ we have that $\phi^{-1}(K)$ is contained in the $1/2^n$-neighborhood  of
    the compact set $(\phi^n_A ,\phi^n_\re)^{-1}(K_A \times [0,k+1])$.
    This implies that $\phi^{-1}(K)$ is compact and hence $\phi$ is proper.

\smallskip

\noindent {\bf Proof of  Theorem~\ref{resolution}.} 
We need to assume that $A$ is a Hilbert type space
  and the only thing  we need
  to verify is that $cone H \approx H$ for a model space  $H$.
  Let $ \pi : H\times [0,1] \lo coneH$ be the projection with $H \times \{ 1 \}$
  being over the vertex.
   Consider $H \times [0,1)$ as
  an open  subspace of $cone H$ and  define the basis of the vertex by the complements
  of $H \times [0,t]$ for $0\leq t < 1$. Then $\pi $ is a near homeomorphism. Indeed,
  fix $\varepsilon >0$,  take  a point
  $x \in U_\varepsilon =H \times (1-\varepsilon, 1)$ and a neighborhood
  $U_x \subset U_\varepsilon$ of $x$.  Approximate the constant map sending $H \times \{ 1\}$
  to $x$  by a  Z-embedding $\phi : H \times \{ 1\} \lo H \times [0,1]$
  such that $\phi (H \times \{ 1 \}) \subset U_x$ and by the Z-set unknotting
  theorem extend $\phi$ to a homeopmorphism $\Phi : H \times [0,1] \lo H \times [0,1]$
  supported by $U_{2\varepsilon}$.
  Now  take a homeomorphism
  $\Psi : H \times [0,1] \lo H \times [0,1]$ supported by $U_{2\varepsilon}$  such that
   $\Psi$  leaves every interval $a \times [0,1], a \in H$
  invariant and  sends $U_\varepsilon$ into $\Phi^{-1}(U_x)$. Then, by Bing's
  shrinking criterion,  $\Phi \circ \Psi^{-1}$ witnesses that $\pi$ is a near homeomorphism.

  \smallskip

\noindent {\bf Remark.} It is also possible to adjust Sections \ref{hilbert-cube-characterization}
 and \ref{cell-like-resolution} for the non-compact setting by replacing cell-like maps
 by fine homotopy equivalences. This approach is more general  but it
 requires more clarifications.

%\def\br#1{d:/Jan/wiskunde/texinputs/#1}

%\bibliographystyle{\br{michael}}
%\bibliographystyle{alpha}
%\bibliography{\br{strings},\br{mill},\br{publ},\br{eric}}

\def\cprime{$'$}
\makeatletter \renewcommand{\@biblabel}[1]{\hfill[#1]}\makeatother

\end{document}